\documentclass[centering,11pt,reqno]{amsart}

\usepackage[utf8]{inputenc}
\usepackage[T1]{fontenc}
\usepackage{amsmath,amsthm}
\usepackage{amsfonts,amssymb}
\usepackage[mathscr]{eucal}
\usepackage{url}
\usepackage{mdframed}
\usepackage{paralist}
\usepackage[colorlinks,cite color=blue,pagebackref=true,pdftex]{hyperref}
\usepackage[margin=2.5cm]{geometry}
\usepackage{tikz}
\usetikzlibrary{calc,shapes}
\usepackage{mathtools}
\usepackage{multicol}
\usepackage{ytableau}
\usepackage{blkarray}
\usepackage[capitalize]{cleveref}
 \usepackage[foot]{amsaddr}
\usepackage{tikz-cd}
\usepackage{ytableau}
\usepackage{bbm}
\usepackage{comment}
\usepackage{color}
\usepackage{todonotes}
\usepackage{accsupp}

\colorlet{mygray}{gray}
\colorlet{mypolyfill}{gray!10}
\colorlet{myblue}{blue}
\colorlet{myred}{red}

\newtheorem{theorem}{Theorem}[section]
\newtheorem*{theorem*}{Theorem}
\newtheorem*{cor*}{Corollary}
\newtheorem*{conj*}{Conjecture}
\newtheorem*{lemma*}{lem}
\newtheorem*{prop*}{Proposition}
\newtheorem{lem}[theorem]{Lemma}
\newtheorem{cor}[theorem]{Corollary}
\newtheorem{prop}[theorem]{Proposition}

\theoremstyle{definition}

\renewcommand{\emptyset}{\varnothing}
\renewcommand{\epsilon}{\varepsilon}
\newcommand{\R}{\mathbb{R}}

\newcommand{\conv}{\mathrm{conv}}

\DeclareMathOperator{\xc}{xc}
\DeclareMathOperator{\vxc}{vxc}
\DeclareMathOperator{\nnc}{nnc}

\def\R{\mathbb{R}}

\def\N{\mathbb{N}}

\newcommand{\defn}[1]{\emph{#1}}

\title{Two Linear Programs are  Better than One}

\author[Black]{Alexander E. Black}
\address{Department of Mathematics, Bowdoin College}
\email{a.black@bowdoin.edu}

\author[Hertrich]{Christoph Hertrich}
\address{Department Liberal Arts and Social Sciences, University of Technology Nuremberg}
\email{christoph.hertrich@utn.de}

\author[Kober]{Stefan Kober}
\address{Department of Mathematics, London School of Economics and Political Science}
\email{S.Kober@lse.ac.uk}

\author[Loho]{Georg Loho}
\address{Institute of Mathematics, Freie Universit\"at Berlin}
\email{georg.loho@math.fu-berlin.de}

\begin{document}

\maketitle

\begin{abstract}  
The extension complexity $\xc(P)$ of a polytope $P$ is the minimum number of facets of a polytope that projects onto it. Motivated by the complexity of neural networks, Hertrich and Loho introduced the virtual extension complexity $\vxc(P)$ of a polytope $P$, which is the minimum of $\xc(Q) + \xc(R)$ among all polytopes $Q$ and $R$ such that $P+Q=R$. We show the following results pertaining to virtual extension complexity. 

First, we resolve the core open problem on the topic by showing the virtual extension complexity can be exponentially smaller than extension complexity. Second, we show there are matroid base polytopes with exponential virtual extension complexity by adapting a semi-algebraic counting argument by Hrube\v{s} that extends the counting arguments of Rothvo{\ss}. 
Finally, we show that the maximum virtual extension complexity of a polygon with $n$ vertices is $\Theta(\sqrt{n})$. This resolves an analogue of a question on extension complexity of Fiorini, Rothvo{\ss}, and Tiwary in 2011 in the case of virtual extension complexity. 
\end{abstract}

\section{Introduction}

Feasible sets of systems of linear inequalities, also known as \emph{polyhedra}, are a fundamental modeling tool for optimization problems.
  The running time of algorithms for solving these problems grows with the number of inequalities. 
  An approach to overcome this obstacle is to use an \emph{extended formulation}, a description in a higher-dimensional space, and then project to the original components.
  This can significantly reduce the required number of inequalities.
  The minimal number of inequalities in such a description is known as \emph{extension complexity}.
  Extended formulations are among our strongest tools for solving combinatorial optimization problems, and as a consequence extension complexity has been studied extensively ~\cite{ExtensionSurvey1, ExplLowBoundsExtensions}.

Hertrich and Loho recently proposed a generalization of extension complexity motivated by neural networks \cite{VirtualExtensions}.  See Subsection \ref{subsec:nncomplexity} for an overview of the connection. The \defn{virtual extension complexity} $\vxc(P)$ of a polytope $P$ is the minimum of the sum $\xc(Q) + \xc(R)$ over all polytopes $Q, R$ such that $P + Q = R$, where $P + Q = \{x + y: x \in P, y \in Q\}$ is the \emph{Minkowski sum}. 
This allows for even more flexibility in the representation of the original optimization problem over~$P$: you solve it over $Q$ and $R$, and then take the difference of optimal values to get the optimal value over $P$. However, there is a trade-off. Minkowski sums generally yield much more complicated polytopes. In fact, at the conclusion of \cite{VirtualExtensions}, they note that there is no known example where extension complexity and virtual extension complexity differ. A priori, they could always be equal. Question 5.1 in \cite{VirtualExtensions} asks whether this is the case.

Our first main result establishes that $\xc$ and $\vxc$ of a polytope can be substantially different.

\begin{theorem}\label{thm:exponential-separation}
Let $n\ge4$ be even, and let $P_{n} \subseteq \mathbb{R}^{\binom{[n]}{2}}$ be the $0/1$-polytope given by the convex hull of indicator vectors of edge sets of the complete graph such that the corresponding spanning subgraph contains a connected component with an odd number of vertices. Then $\xc(P_{n}) \ge 2^{\Omega(n)}$ and $\vxc(P_{n}) = O(n^{8})$.
\end{theorem}

To prove this, we draw on classical results in combinatorial optimization. In breakthrough work, Rothvo\ss{}~\cite{MatchingExtensions} introduced a powerful technique to show that the matching polytope has exponential extension complexity.
The family of polytopes $P_n$ is sufficiently similar to the matching polytope for Rothvo\ss{}' technique to still apply to give exponential lower bounds on their extension complexity. In contrast, we prove that it has polynomial virtual extension complexity by using the connection to neural networks for representing min-cut problems from   \cite{hertrich2025relu}.

This gives hope that virtual extension complexity could be a better model for polynomial time solvability. Could it be that most popular polynomial time solvable problems admit polynomial sized virtual extended formulations? We study this question next.

To do this, we consider $0/1$-polytopes, that is, those polytopes whose vertices have only components $0$ or $1$.
Restricting edge directions to differences of standard basis vectors in $\R^n$ yields \defn{matroid base polytopes}, a class of particular importance in combinatorial optimization.
Upper bounds are known for \emph{regular matroids}, which can be defined as those matroids that can be represented over any field. 
More precisely, by~\cite{AF22}, their extension complexity is upper bounded by $O(n^6)$, while their virtual extension complexity has been shown in~\cite{HKL26} to be upper bounded by $O(n^3)$. 
Concerning lower bounds, using a counting-based argument, Rothvo\ss~\cite{Rot13} showed that most 0/1 polytopes, and even almost all matroid base polytopes, have exponential extension complexity. 

Using the idea of counting those subsets of $\{0,1\}^n$ expressible by real-algebraic formulas, Hrube\v{s}~\cite{Hru20} derived a slightly weaker bound though with a more flexible framework.
In Section~\ref{sec:counting-lower-bound}, we show how this technique can be applied to 
obtain lower bounds on the \textbf{virtual} extension complexity of almost all matroid base polytopes. 
The following is a direct consequence of Theorem~\ref{thm:counting} using that there are doubly-exponentially many matroid base polytopes~\cite{Dukes2003}.

\begin{theorem} \label{thm:mbp} 
    There exists $\varepsilon'>0$ such that asymptotically almost every matroid $M$ on the ground set $[n]$ satisfies $\vxc(B(M))\ge2^{\varepsilon'n}$, where $B(M)$ is the base polytope of $M$. 
\end{theorem}

Thus, even classical problems that are solvable in polynomial time may still have no polynomial sized virtual extended formulation.

Finally, we focus on the case of polygons, which has received substantial attention in the extension complexity literature. The core question, asked originally in 2011 in \cite{PolygonExtensions}, is what the worst-case extension complexity of a polygon with $n$ vertices is, and they prove a lower bound of $\Omega(\sqrt{n})$ and no nontrivial upper bound. The first upper bounds less than $n$ of roughly $\lceil 6n/7 \rceil$ were attained independently by Shitov \cite{ShitovSubN} and Padrol and Pfeifle \cite{PadrolPfeifleSubN}. Shitov then improved this upper bound to a first sublinear bound of $o(n)$ in a preprint and then later strengthened this to an upper bound of $O(n^{2/3})$ in \cite{shitovsublinear}. In the case of random polygons, combined results of Padrol \cite{RandomPolygonLowerBounds} and Kwan, Sauermann, and Zhao  \cite{LowDimExtensions} imply that the expected extension complexity of a random polygon with $n$ vertices is $\Theta(\sqrt{n})$. Furthermore, for cyclic polygons, i.e., polygons inscribed in a circle, Kwan, Sauermann, and Zhao showed in the same work that their maximum extension complexity is $O(\sqrt{n})$ and conjectured that this should be the case for all polygons. 

Motivated by their work, we asked the same question for virtual extension complexity, and managed a complete resolution via a geometric argument extending the ideas from \cite{LowDimExtensions}: 

\begin{theorem}
\label{thm:main}
The maximum virtual extension complexity of a polygon with $n$ vertices is $\Theta(\sqrt{n})$. 
\end{theorem}

It remains an open question whether in the case of polygons, virtual extension complexity can be significantly smaller than extension complexity. If not, then our result implies that the extension complexity of a polygon will be at most $O(\sqrt{n})$. We propose this as a novel approach to that conjecture. 

A further significant question left open by our work is to determine the worst-case asymptotic virtual extension complexity of the matching polytopes. Furthermore, considering the original motivation of $\vxc$ via neural networks, it remains an open question whether the support function of the matching polytope can be represented by polynomial size neural networks. This stronger question is even open for the support function of the Birkhoff polytope, corresponding to maximum-weight bipartite perfect matching.

\subsection{Preliminaries}

\subsubsection{Extended formulations} Let $P\subseteq\R^n$ be a polyhedron. 
We say that a polyhedron $Q$ in $\R^d$ is an \defn{extension} of $P$, if there is a projection $\pi:\mathbb{R}^d\to\mathbb{R}^n$ such that $\pi(Q)=P$.
More formally, we have $A\in\R^{m\times d}$, $b\in\R^{m}$, $C\in\R^{n\times d}$ and $c\in\R^n$ such that 
\[
    Q=\{y\in\R^d: Ay\le b\}\qquad\text{and}\qquad P=\{Cy+c: y\in Q\}.
\]
Then, $(A,b,C,c)$ is called an \defn{extended formulation} of $P$.
Note that if $P$ is a polytope, then we can assume without loss of generality that $d\le m$.
The \defn{extension complexity} $\xc(P)$ is defined as the smallest number of inequalities (number of rows of $A$) of any extended formulation of $P$ (note that we do not use equations in the definition of extended formulations, since they can be projected out). 

\subsubsection{Virtual extended formulations} 
Recall the \defn{Minkowski sum} of two subsets $X$ and $Y$ of $\R^n$ denoted by $X+Y:=\{x+y:x\in X,y\in Y\}$.
While we just defined extension complexity more generally for (potentially unbounded) polyhedra, virtual extension complexity is only defined for polytopes. See the conclusion of Section 4 of \cite{VirtualExtensions} for discussion as to why. 
The \defn{virtual extension complexity} of a polytope $P\subseteq\R^n$ is 
\[\vxc(P):=\min\left\{\xc(Q)+\xc(R): Q, R \text{ polytopes with } P+Q=R; \right\}.\]

Linear optimization over a polytope $P\subseteq\R^n$ can equivalently be phrased as evaluating the \defn{support function}, that is, the function $h_P(c):=\max\{c^{\top}x:x\in P\}$ for a fixed $c \in \R^n$.

For polytopes $P,Q,R\subseteq\R^n$ with $P+Q=R$, we have for all $c\in\R^n$ that $h_P(c)+h_Q(c)=h_R(c)$.
This shows that indeed linear optimization over $P$ can be reduced to linear optimization over extended formulations of $Q$ and $R$; see also \cite{VirtualExtensions} for details on how to recover the actual optimal solution. 
Hence, virtual extension complexity is actually a reasonable complexity measure for linear optimization over~$P$.

\subsubsection{Neural network complexity}
\label{subsec:nncomplexity}
For a general introduction to neural networks, we refer to \cite{Goodfellow-et-al-2016}. We use the same definitions and conventions as in~\cite{VirtualExtensions}.
In particular, when talking about neural networks, we generally refer to \emph{rank-2 maxout networks} as defined in~\cite{VirtualExtensions}. Up to a small constant factor in size, every such network is equivalent to the well-known model of \emph{ReLU feedforward networks}.

A \defn{maxout neural network} is based on a directed, acyclic graph with node set $V$ and arcs $A \subseteq V \times V$, where the nodes are called \defn{neurons} and act as computational units.
The $d\geq1$ neurons with in-degree zero are called \defn{input neurons}. All other $s\geq1$ neurons are \defn{maxout units}. Every maxout unit computes the maximum of two linear combinations of the outputs of its predecessors.
We assume that among the maxout units there exists exactly one neuron with out-degree zero, the \defn{output neuron}.
The \defn{size} of the neural network is defined as the number $s$ of maxout units.

Note that we define neural networks without biases. For our purposes, this is without loss of generality, compare \cite[Remark~2.1]{VirtualExtensions}.
The functions represented by such maxout neural networks without biases are exactly the positively homogeneous CPWL (continuous piecewise linear) functions~\cite{Arora2018}. Note, however, that the required number of neurons can be large.
On the other hand, positively homogeneous CPWL functions are exactly the differences of support functions of polytopes.
For a polytope $P$, we define the \defn{neural network complexity} $\nnc(P)$ as the minimum number of neurons to represent its support function $h_P$ by a maxout neural network.

Hertrich and Loho introduced this notion and showed in~\cite[Theorem~4.1]{VirtualExtensions} the following inequality.
\begin{prop} \label{prop:vxc-nnc}
    For any polytope $P$, $\vxc(P)\le4\cdot\nnc(P)$.
\end{prop}

As mentioned above, while we introduced neural network complexity using maxout networks, this carries over to ReLU-networks with the cost of a small constant factor~\cite[Remark~2.2]{VirtualExtensions}.

We use Proposition~\ref{prop:vxc-nnc} to prove the upper bound on virtual extension complexity in Section~\ref{sec:exponential-separation}, where we construct a neural network to represent the support function of the considered polytope~$P_n$. Furthermore, combining Proposition~\ref{prop:vxc-nnc} with our counting-based lower bound for matroid polytopes (Theorem~\ref{thm:mbp}), we obtain the following corollary.

\begin{cor}\label{cor:nnc}
    There exists $\varepsilon''>0$ such that asymptotically almost every matroid $M$ on the ground set $[n]$ satisfies $\nnc(B(M))>2^{\varepsilon'' n}$, where $B(M)$ is the base polytope of $M$.

    Hence, there is a family of matroids for which every ReLU network exactly computing the maximum-weight basis value has size $2^{\Omega(n)}$.
\end{cor}

\section{Exponential Separation}
\label{sec:exponential-separation}

We organize the proof of Theorem~\ref{thm:exponential-separation} into two parts.
The proof of the lower bound on the ordinary extension complexity follows from the work of Walter and Weltge~\cite{WalterWeltgeRadialCones} on the odd-cut dominant, building on Rothvo{\ss}'s exponential lower bound for the perfect matching polytope~\cite{MatchingExtensions}.
For the upper bound on the virtual extension complexity, we express the support function of the polytope through certain minimum cut values.
Then, we build a polynomial size neural network which exactly computes this expression.
This yields an explicit virtual extended formulation of size $O(n^8)$ by Proposition \ref{prop:vxc-nnc}.

Throughout this section, let $n\ge4$ be even, $V=[n]$, $E=\binom{V}{2}$ and $m=|E|$.
Recall that the polytope $P_n$ defined in the introduction is the convex hull of the incidence vectors of subsets of edges of the complete graph on $n$ vertices, such that the corresponding spanning subgraph contains a connected component with an odd number of vertices, that is,
\[P_n:=\conv\{\chi^{F}:F\subseteq E,\ (V,F)\text{ has a component of odd size}\}.\]
For a graph $G=(V,F)$ and a subset of the vertices $S\subseteq V$, let $\delta_G(S)$ denote the edges with exactly one endpoint in $S$.
If the graph is the complete graph on $n$ vertices, we also write $\delta(S):=\delta_{K_n}(S)$.
For $c\in\R_+^E$ and $F\subseteq E$, write $c(F)=\sum_{e\in F}c_e$.
Further, for distinct $u\neq v\in V$ we denote the minimum weight cut separating $u$ from $v$ by $\lambda_{uv}(c)$, and the overall minimum weight of an odd cut by $\tau(c)$.
To be more precise, we define 
\[
    \lambda_{uv}(c):=\min_{\substack{S\subseteq V\\u\in S,\ v\notin S}}c(\delta(S)),
    \qquad
    \tau(c):=\min_{\substack{S\subseteq V\\|S|\text{ odd}}}c(\delta(S)).
\]
For a polytope $K$, its support function is $h_K(w)=\max_{x\in K}\langle w,x\rangle$.

We will require the following elementary observation on $P_n$, which allows us to reduce to non-negative objective functions.
\begin{lem}\label{lem:monotone}
    A graph $G=(V,F)$ has a connected component of odd size if and only if
    $\delta_G(S)=\emptyset$ for some odd-cardinality set $S\subseteq V$.

    Moreover, let $w\in\R^E$ be a cost function on the edges, and let $c_e=\max\{w_e,0\}$ be its non-negative part.
    Then
    \[
        h_{P_n}(w)=h_{P_n}(c)=c(E)-\tau(c).
    \]
\end{lem}
\begin{proof}
    If $G$ has a component of odd size with vertex set $S$, then $\delta_G(S)=\emptyset$.
    Conversely, if $\delta_G(S)=\emptyset$, then $S$ is a union of connected
    components of $G$.
    Since $S$ contains an odd number of vertices, at least one of these components has odd size.
    Since we have non-negative edge weights, it is optimal to include as many edges as possible in the solution.
    Hence, we obtain that 
    \[
        h_{P_n}(c)=\max_{\substack{S\subseteq V\\|S|\text{ odd}}}
          c(E\setminus\delta(S))
        =c(E)-\tau(c).
    \]
    
    It remains to show that $h_{P_n}(w)=h_{P_n}(c)$.
    Indeed, let $F\subseteq E$ be a set of edges such that $(V,F)$ contains a connected component of odd size, and $c(F)$ is maximized.
    Clearly, any $F'\subseteq F$ still induces a connected component of odd size.
    Hence, we can remove all edges with a negative $w$-weight from $F$ and obtain a set $F'$ with $w(F')=c(F)$.
    Since $F$ was chosen to maximize $c(F)$, this implies $h_{P_n}(w)=h_{P_n}(c)$.
\end{proof}
We remark that similar conditions have been considered before.
In particular, the polytope $P_n$ is also monotone in the sense of Rothvo{\ss}~\cite{MatchingExtensions}.
For $x\in P_n$ and $0\le y\le x$, we have that $y\in P_n$.

\subsection{The lower bound on the extension complexity}
We derive the lower bound on the extension complexity of $P_n$ via the odd-cut dominant $D_n$, which is defined as
\[
    D_n=\conv\{\chi^{\delta(S)}:S\subseteq V,\ |S|\text{ odd}\}+\R_+^E.
\]
Walter and Weltge~\cite{WalterWeltgeRadialCones}
show that $\xc(D_n)\ge2^{\Omega(n)}$, by using Rothvo{\ss}'s lower bound for the
perfect matching polytope~\cite{MatchingExtensions}.
Note that in the context of virtual extended formulations, we only work with bounded polytopes. 
We only use the odd-cut dominant for the sake of presenting a simple proof of the lower bound on $\xc(P_n)$.
Alternatively, one could derive the statement directly from~\cite{MatchingExtensions} using some more technical arguments.

\begin{prop}\label{prop:odd-ordinary-complexity}
    Let $n\ge4$ be even. Then, $P_n$ satisfies $\xc(P_n)\ge2^{\Omega(n)}$.
\end{prop}
\begin{proof}
    Let $F\subseteq E$ be a set of edges such that $(V,F)$ contains a connected component of odd size.
    By Lemma~\ref{lem:monotone}, there exists a set of vertices $S$ of odd size such that $\mathbf{1}-\chi^F\ge\chi^{\delta(S)}$. 
    Moreover, for each set of vertices of odd size $S$, there is such a set of edges $F=E\setminus\delta(S)$, such that $\mathbf{1}-\chi^F=\chi^{\delta(S)}$.
    Hence, we can add the non-negative orthant $\R_+^E$ to $(\mathbf{1}-P_n)$ and obtain that
    \[
        D_n=(\mathbf{1}-P_n)+\R_+^E.
    \]

    Hence, if we have an extended formulation for $P_n$, it suffices to add auxiliary variables $z\in\R_+^E$ and to change the projection to invert the edge coordinates and add the $z$-coordinates, i.e., we project to $\mathbf{1}-x+z$.
    This process adds $m$ inequalities for the non-negativity condition of the $z$-variables.
    Thus, we obtain that $\xc(P_n)\ge\xc(D_n)-m\ge2^{\Omega(n)}$.
\end{proof}

\subsection{Minimum odd cuts from pairwise cuts}
In this section, we derive a formula for the minimum odd cut $\tau(c)$ based on the pairwise minimum cut values $\lambda_{uv}(c)$ for distinct vertices $u\neq v\in V$.
Our argument is based on a classical theorem of Gomory and Hu~{\cite{GomoryHuTrees}}, see Theorem~\ref{thm:gomory-hu}.

\begin{theorem}[Gomory--Hu~{\cite{GomoryHuTrees}}]\label{thm:gomory-hu}
    For every $c\in\R_+^E$, there is a tree $T$ on the vertex set $V$ and a
    weight function $c':E(T)\to\R_+$ such that, for distinct $u,v\in V$,
    \[
        \lambda_{uv}(c)=\min\{c'(e):e\in T_{uv}\},
    \]
    where $T_{uv}$ is the set of edges on the unique $u$--$v$ path in $T$.
    Moreover, deleting an edge $e$ of $T$ partitions $V$ into the vertex sets $U$ and $U'$ of the two corresponding subtrees, with $c(\delta(U))=c'(e)$.
\end{theorem}

There is a related threshold graph $H_{t}(c)$ to the Gomory-Hu tree, where we say two vertices $u$ and $v$ are adjacent if their min $u-v$ cut is of weight at least $t$. We first understand this graph and relate it to $\tau(c)$.

\begin{lem}\label{lem:odd-threshold-components}
    For $c\in\R_+^E$ and $t\ge0$, define $H_t(c)=\bigl(V,\{uv\in E:\lambda_{uv}(c)\ge t\}\bigr)$.
    Every connected component of $H_t(c)$ is a clique.
    All its connected components have even order if and only if $t\le\tau(c)$.
\end{lem}
\begin{proof}
    Choose $T$ and $c'$ as in Theorem~\ref{thm:gomory-hu}.
    Thus, we obtain the components of $H_t(c)$ by deleting edges of $T$ of weight less than $t$ and completing each remaining component of $T$ to a clique.

    If $t\le\tau(c)$, every edge that was deleted from $T$ corresponds to a cut of weight less than $\tau(c)$.
    Hence, both sides of the cut have even size.
    Consider any component $C$ that remains connected after deleting the edges from $T$.
    Its incident deleted tree edges induce separations such that the sides disjoint from $C$ partition $V\setminus C$.
    By the previous argument, each of these sets has even size.
    Since $n$ is even, this implies that also $C$ has even size.

    If $t>\tau(c)$, choose a subset of vertices $S$ with odd size such that $c(\delta(S))=\tau(c)$.
    For $u\in S$ and $v\notin S$, we have $\lambda_{uv}(c)\le\tau(c)<t$.
    Thus, no edge of $H_t(c)$ crosses $S$.
    Since $|S|$ is odd, at least one of the components contained in $S$ has odd order.
\end{proof}

Our goal is to get an expression for $\tau(c)$ in terms of minimum cut values. 
To do this, we measure how the size of a component containing a given vertex in $H_{t}(c)$ changes as $t$ increases. 
A similar idea for an incremental construction was used to solve the odd minimum cut-set problem using Gomory-Hu trees~\cite{PadbergRao:1982}. 
We use it to derive a simple powerful expression in terms of pairwise cut values. 

\begin{lem}\label{lem:odd-parity-identity}
    Let $c\in\R_+^E$. For each $v\in V$, choose a bijection
    $\pi_v:[n-1]\to V\setminus\{v\}$ such that $\lambda_{v\pi_v(1)}(c)\le\cdots\le\lambda_{v\pi_v(n-1)}(c)$.
    Then
    \[
        \tau(c)=\min_{v\in V}\sum_{j=1}^{n-1}(-1)^{j+1}\lambda_{v\pi_v(j)}(c).
    \]
\end{lem}
\begin{proof}
    Let $C_v(t)$ be the component of $H_t(c)$ containing $v$.
    By Lemma~\ref{lem:odd-threshold-components}, it is a clique, so
    \[
        |C_v(t)|=1+\bigl|\{u\ne v:\lambda_{uv}(c)\ge t\}\bigr|.
    \]
    Since $n-1$ is odd, the thresholds $t>0$ for which the size of $C_v(t)$ is even form the disjoint intervals
    \[
        (0,\lambda_{v\pi_v(1)}(c)],\quad
        (\lambda_{v\pi_v(2)}(c),\lambda_{v\pi_v(3)}(c)],\quad\ldots,
    \]
    using the permutation of the cut values from the statement of the Lemma.
    Their total length is $\sum_{j=1}^{n-1}(-1)^{j+1}\lambda_{v\pi_v(j)}(c)$.
    By Lemma~\ref{lem:odd-threshold-components}, the union of the intervals covers
    $(0,\tau(c)]$, for any vertex $v\in V$, so this sum is at least $\tau(c)$.

    It remains to show that there exists some vertex $v^*\in V$ such that $\tau(c)=\sum_{j=1}^{n-1}(-1)^{j+1}\lambda_{v^*\pi_{v^*}(j)}(c)$.
    Since there are only finitely many cut values, choose $\varepsilon>0$
    such that $H_t(c)$ is constant for
    $\tau(c)<t\le\tau(c)+\varepsilon$.
    By Lemma~\ref{lem:odd-threshold-components}, $H_t(c)$ contains a component of odd size. 
    We choose one such component and denote it by $S$.
    As $t$ increases, components only split and eventually become singletons.
    Whenever $S$ splits, we maintain a component of odd size.
    When this process terminates, we obtain a single vertex $v^*$.
    This vertex therefore belongs to a component of odd size for every $t>\tau(c)$.
    Hence, the union of the corresponding intervals for which the size of $C_{v^*}(t)$ is even has length exactly $\tau(c)$.
\end{proof}

\subsection{Construction of the neural network}

The power of Lemma \ref{lem:odd-parity-identity} is that it reduces computing $\tau(c)$ into the following steps. First compute pairwise cuts. Then sort them. Then compute minimum of their alternating sums. All of this fits well into a neural network representation.

\begin{proof}[Proof of Theorem~\ref{thm:exponential-separation}]
	The lower bound on $\xc(P_n)$ is established by Proposition~\ref{prop:odd-ordinary-complexity}. 
    To prove the upper bound, we use the connection between $\vxc$ and neural network complexity from Proposition~\ref{prop:vxc-nnc}. 
	To prove that $\vxc(P_n)\in O(n^8)$, it therefore suffices to show that $h_{P_n}$ can be represented by a neural network of size $O(n^8)$. For that, we use Lemma~\ref{lem:monotone} to obtain $h_{P_n}(w)=c(E)-\tau(c)$. From the inputs $w(e)$, it is straightforward to compute the values $c(e)$ and the sum $c(E)$. It therefore remains to prove that a neural network of size $O(n^8)$ can compute the value $\tau(c)$ from the values~$c(e)$.
	
	To prove this, we use the expression in Lemma~\ref{lem:odd-parity-identity}. We first show that each of the values $\lambda_{uv}$, for two distinct vertices $u,v\in V$, can be computed by a neural network of size $O(n^6)$. By the Max-flow min-cut theorem, $\lambda_{uv}(c)$ equals the maximum flow that can be sent from $u$ to $v$ in the directed graph that is obtained by replacing each undirected edge by a pair of directed edges in opposite directions with the same capacities. On a graph with $n$ vertices and $a$ arcs, Hertrich and Sering~\cite{hertrich2025relu} showed that this can be done with a neural network of size $O(n^2a^2)\subseteq O(n^6)$. Hence, the same is true for each of the values $\lambda_{uv}$. Doing this for all pairs of vertices $u,v$ yields a total size of $O(n^8)$.
	The next step towards representing $\tau(c)$ with a neural network is for each $v\in V$ to sort the values $\lambda_{uv}$ for $u\neq v$ in increasing order. A pairwise comparison can be implemented by a constant-size neural network~\cite{Arora2018}, using the identity
    \[
        (\min\{a,b\},\max\{a,b\})
        =\bigl(a-\max\{a-b,0\},\,b+\max\{a-b,0\}\bigr).
    \]
    The values $\lambda_{uv}$ can therefore be sorted by mimicking any sorting network~\cite{Batcher1968}. This is not the bottleneck in our complexity estimate. Hence, the reader may, for example, think of the values as being sorted by a hard-wired bubble sort.
    We need to sort $O(n)$ lists with $O(n)$ entries each, which is possible with size $O(n(n^2))=O(n^3)$.
	Once we have the sorted values $\lambda_{uv}$ available, the formula for $\tau(c)$ in Lemma~\ref{lem:odd-parity-identity} is only a predefined min-sum expression, which can be directly implemented in a neural network. It can be realized with $O(n^2)$ additions and $O(n)$ minimum computations and therefore does not further increase the complexity estimate.	
\end{proof}

We remark that we believe that it is possible to obtain a stronger upper bound on the virtual extension complexity, that is, $\vxc(P_n)\le O(n^4)$. 
We note that proving the bound of $O(n^4)$ for the neural network complexity is likely to require new ideas.
We decided to present the above proof of the weaker bound via neural networks, since it is technically less involved and shows the exponential separation also for neural networks.

\section{Counting Lower Bounds}
\label{sec:counting-lower-bound}

We start by defining first-order formulas over the real numbers.
An \defn{atomic formula} is an equality or inequality between (multi-variate) polynomials with real coefficients, i.e., $p(x)=q(x)$, $p(x)\le q(x)$ or $p(x)<q(x)$ for polynomials $p$ and $q$.
A general \defn{formula} can be obtained from atomic formulas by using quantifiers and the usual Boolean connectives.
We define the \defn{size} of a formula as the number of symbols in the formula (constants and variables having a unit cost).

A formula is \defn{quantifier-free}, if it contains no quantifiers.
A \defn{$\Sigma_1$-formula} is a formula of the form $G(x):=\exists z_1,\dots,\exists z_k F(x,z)$, where $F(x,z)$ is a quantifier-free formula.
Let $F(x_1, \dots, x_n)$ be a formula with no free variables other than $x_1, \dots, x_n$. 
We will say that $F$ defines a Boolean function $f : \{0, 1\}^n \to \{0, 1\}$ if the following holds:
\[f (\sigma_1,\dots,\sigma_n) = 1\ \Longleftrightarrow\ F(\sigma_1,\dots,\sigma_n) \quad\text{is true, for every}\quad \sigma_1,\dots,\sigma_n \in \{0, 1\}.\]
We point out that the above definitions extend to higher order and arbitrary fields, but this is not necessary for the results presented below.

The main idea of the proof is to use a virtual extended formulation $Q,R$ of a polytope $P$ in order to express non-membership in $P$ via an existential first-order formula whose size is polynomial in the extension complexities of $Q$ and $R$.
We achieve this in Proposition~\ref{prop:lb} by analyzing non-membership through the Minkowski sum, and applying a variant of Farkas' lemma.
The central complexity-theoretic tool is a theorem of Hrube\v{s}~\cite{Hru20}, see Theorem~\ref{Hrubes}, which bounds the number of Boolean functions which can be described via small existential first-order formulas.
We complete the proof by observing that 0/1-polytopes uniquely define Boolean functions, which means that we can apply the bound from Theorem~\ref{Hrubes}.

\begin{theorem}\label{thm:counting}
    There exists some constant $c>0$, such that for all $n,s\in\N$ with $s\ge n$, the number of 0/1-polytopes in $\R^n$ of virtual extension complexity at most $s$ is at most $2^{s^c}$.
\end{theorem}

Since the number of 0/1-polytopes as well as the number of matroid base polytopes is doubly exponential
in the ambient dimension \cite{Dukes2003}, this yields the following corollary and concludes the proof of Theorem~\ref{thm:mbp}.

\begin{cor}\label{cor:0-1}
    There exists $\varepsilon>0$, such that 
    \[\frac{\#\{X\subseteq\{0,1\}^n:\vxc(\mathrm{conv}(X))\le 2^{\varepsilon n}\}}{2^{2^n}}\longrightarrow0.\]

    Hence, asymptotically almost all 0/1-polytopes have virtual extension complexity $2^{\Omega(n)}$.
\end{cor}

\subsection{Not so many polytopes}
The following lemma allows us to formulate comparisons between support functions of polytopes as a system of inequalities based on their respective extended formulations.

\begin{lem}\label{lem:certificates}
    Let $Q,R\subseteq\R^n$ be two polytopes with extended formulations $(A_Q,b_Q,C_Q,c_Q)$ and $(A_R,b_R,C_R,c_R)$, that is, 
    \[
    Q=\{C_Qu+c_Q:A_Qu\le b_Q\}\qquad\text{and}\qquad R=\{C_Rv+c_R:A_Rv\le b_R\}
    \]
    respectively.
    Let $c\in\R^n$ and $\alpha\in\R$.
    Then, the inequality $\alpha+h_Q(c)>h_R(c)$ holds if and only if there exist vectors $u$ and $t\ge 0$ of appropriate dimensions such that
    \[A_Qu\le b_Q,\quad A_R^{\top}t=C_R^{\top}c,\quad \text{and}\quad \alpha+c^{\top}(C_Qu+c_Q)>b_R^{\top}t+c_R^{\top}c.\]
\end{lem}
\begin{proof}
    We note that we can express $h_Q(c)$ by a primal existence certificate, i.e., there exists a vector $u$ such that $A_Qu\le b_Q$, and $c^{\top}(C_Qu+c_Q)=h_Q(c)$.
    In the same way, we can formulate the dual of the linear program $\max\{c^{\top}(C_Rv+c_R):\ A_Rv\le b_R\}$, which is given by $\min\{b_R^\top t+c_R^\top c:\ A_R^\top t=C_R^\top c,\ t\ge0\}$.
    Hence, by strong duality there exists some $t\ge0$ such that $A_R^{\top}t=C_R^{\top}c$ and $b_R^{\top}t+c_R^{\top}c= h_R(c)$.
    Conversely, we can use weak duality on $R$ and project $u$ to $Q$ in order to see that the existence of $u$ and $t$ implies the strict inequality.
\end{proof}

In Proposition~\ref{prop:lb} we show that if the virtual extension complexity of a polytope $P$ is bounded, then there exists a small existential $\Sigma_1$-formula which certifies that a point $x$ is not in $P$. 
Recall that $P+Q=R$.
In order to certify non-membership of $x$ in $P$, we can equivalently give a point $q\in Q$, as well as a dual certificate $t$ for $R$, showing that $x+q$ is not in $R$.
In order to encode the dual certificate $t$, we need to align it with the existence of a separating hyperplane between $P$ and $x$, described by an orthogonal vector $d$, see also Figure~\ref{fig:lb}.
Note that we can only obtain an existential first-order formula for \emph{non-membership} in $P$ like this, since we can ask for the existence of a separating hyperplane.
If we want to encode \emph{membership} in $P$ instead, the natural variant is to use a \emph{for-all}-quantifier, which does not lead to an existential first-order formula.

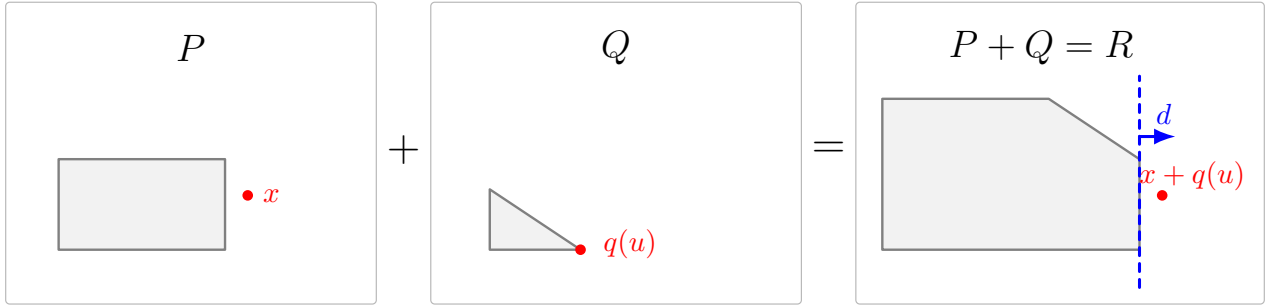
\begin{figure}[!th]
    \begin{tikzpicture}[x=1cm,y=1cm, line cap=round, line join=round, >=Latex]
        
        \tikzset{
          panel/.style={draw=mygray!55, rounded corners=1.5pt, line width=0.45pt},
          poly/.style={draw=mygray, fill=mypolyfill, line width=0.9pt},
          cert/.style={draw=myblue, line width=1.15pt},
          title/.style={font=\Large\itshape},
          setlabel/.style={font=\Large\itshape, text=mygray},
          lab/.style={font=\normalsize\itshape},
          smalllab/.style={font=\small\itshape},
          tinylab/.style={font=\scriptsize\itshape},
          certlabel/.style={font=\small, text=mygray}
        }
        
        \def\W{4.9}
        \def\H{4.0}
        \def\G{0.72}
        \coordinate (Ppanel) at (0,0);
        \coordinate (Qpanel) at ($(Ppanel)+(\W+\G,0)$);
        \coordinate (Rpanel) at ($(Qpanel)+(\W+\G,0)$);
        
        \draw[panel] (Ppanel) rectangle ++(\W,\H);
        \draw[panel] (Qpanel) rectangle ++(\W,\H);
        \draw[panel] (Rpanel) rectangle ++(\W+0.5,\H);
        
        \node[title] at ($(Ppanel)+(\W/2,\H-0.60)$) {$P$};
        \node[title] at ($(Qpanel)+(\W/2,\H-0.60)$) {$Q$};
        \node[title] at ($(Rpanel)+(\W/2,\H-0.60)$) {$P+Q=R$};
        
        \node[font=\LARGE] at ($(Ppanel)+(\W+0.36,2.05)$) {$+$};
        \node[font=\LARGE] at ($(Qpanel)+(\W+0.36,2.05)$) {$=$};

        \begin{scope}[shift={($(Ppanel)+(0.70,0.72)$)}]
          \filldraw[poly] (0,0) rectangle (2.2,1.2);
        
          \fill[myred] (2.50,0.72) circle (0.07);
          \node[lab, text=myred, anchor=west] at (2.57,0.72) {$x$};
        \end{scope}
        
        \begin{scope}[shift={($(Qpanel)+(0.78,0.72)$)}]
          \filldraw[poly] (0,0) -- (1.2,0) -- (0,0.8) -- cycle;
        
          \fill[myred] (1.2,0) circle (0.07);
          \node[lab, text=myred, anchor=west] at (1.37,0.08) {$q(u)$};
        \end{scope}
        
        \begin{scope}[shift={($(Rpanel)+(0.35,0.72)$)}]
          \filldraw[poly] (0,0) -- (3.4,0) -- (3.4,1.2) -- (2.2,2.0) -- (0,2.0) -- cycle;
        
          \draw[cert,dashed] (3.4,-0.50) -- (3.4,2.3);
          \draw[cert,->] (3.4,1.50) -- +(0.48,0);
          \node[lab, text=myblue, anchor=south] at (3.72,1.53) {$d$};
        
          \fill[myred] (3.70,0.72) circle (0.07);
          \node[lab, text=myred] at (4.1,0.98) {$x+q(u)$};
        \end{scope}

    \end{tikzpicture}
    \caption{Geometric interpretation of the formula in Proposition~\ref{prop:lb}. 
    A feasible vector $u$ yields $q(u):=C_Qu+c_Q\in Q$. 
    The dual certificate $t$, together with the normal $d$, certifies a valid halfspace for $R$ that is violated by $x+q(u)$. 
    Hence $x\notin P$.}
    \label{fig:lb}
\end{figure}

\begin{prop}\label{prop:lb}
    Let $s,n\in\mathbb{N}$ be fixed.
    For every polytope $P\subseteq\mathbb{R}^n$ with $\vxc(P)\le s$ 
    there is an existential $\Sigma_1$-formula $F_P(x_1,\dots,x_n)$ of size $O(s(s+n))$ which is true exactly when $x\notin P$.
\end{prop}
\begin{proof}
    Let $Q,R\subseteq\mathbb{R}^n$ be a virtual extended formulation of $P$, i.e., $Q,R$ are polytopes such that $P+Q=R$ and $\xc(Q)+\xc(R)\le s$.
    We set $m_Q:=\xc(Q)$ and $m_R:=\xc(R)$.
    As noted above, we can assume that the representations of $Q$ and $R$ fulfill $d_Q\le m_Q$ and $d_R\le m_R$.
    Then, there are extended formulations $(A_Q,b_Q,C_Q,c_Q)$, $(A_R,b_R,C_R,c_R)$ for $Q$ and $R$, that is, 
    \[
    Q=\{C_Qu+c_Q:A_Qu\le b_Q\}\qquad\text{and}\qquad R=\{C_Rv+c_R:A_Rv\le b_R\}
    \]
    respectively, where $A_Q\in\R^{m_Q\times d_Q}$, $b_Q\in\R^{m_Q}$, $C_Q\in\R^{n\times d_Q}$, $c_Q\in\R^{n}$, and $A_R\in\R^{m_R\times d_R}$, $b_R\in\R^{m_R}$, $C_R\in\R^{n\times d_R}$, $c_R\in\R^{n}$.

    If $x\notin P$, then there is a separating hyperplane, i.e., there exists $d\in\R^n$ such that $d^{\top}x>h_P(d)$, or equivalently $d^{\top}x+h_Q(d)>h_R(d)$.
    We apply Lemma~\ref{lem:certificates} with $c:=d$ and $\alpha:=d^{\top}x$ and obtain that $x\notin P$ if and only if there exist $d\in\R^n$, $u\in\R^{d_Q}$, and $t\in\R^{m_R}$ such that
    \[
        A_Qu\le b_Q,\quad t\ge0,\quad A_R^{\top}t=C_R^{\top}d,\quad \text{and}\quad d^{\top}x+d^{\top}(C_Qu+c_Q)>b_R^{\top}t+c_R^{\top}d.
    \]
    This is the desired existential $\Sigma_1$-formula, which we denote by $F_P(x)$.
    The size of $F_P(x)$ is clearly dominated by the size of the matrices $A_Q,A_R,C_Q$ and $C_R$, giving the desired bound.
    We remark that $d^{\top} C_Q u$ is bilinear in the existential variables, which is allowed in the context of formulas as defined here.
\end{proof}  

The main lower bound result driving our reduction is Theorem~\ref{Hrubes}, which is proved by Hrube\v{s} in~\cite[Theorem~2.2]{Hru20}. 
It essentially states that there are only $2^{r^{O(1)}}$ many Boolean functions that can be defined by $\Sigma_1$-formulas of size at most $r$.
Since the statement of Theorem~\ref{Hrubes} does not appear verbatim in~\cite{Hru20}, we proceed to outline how to obtain this statement from Lemma~\ref{Hrubes-lemma} below, by following the arguments in the proof of~\cite[Theorem~2.2]{Hru20}. 

\begin{lem}[Hrube\v{s}~{\cite[Lemma~5.1]{Hru20}}]\label{Hrubes-lemma}
    Let $F_1(y),\dots,F_N(y)$ be $\Sigma_1$-formulas each of size at most $r$, such that their free variables are $y_1,\dots,y_n$.
    Then, 
    \[\#\{(F_1(a),\dots,F_N(a)):a\in\R^n\}\le\left(2^{r^{O(1)}}N\right)^n.\]
\end{lem}

\begin{theorem}[Hrube\v{s}~{\cite[proof of Theorem~2.2]{Hru20}}]\label{Hrubes}
    Let $n,r\in\mathbb{N}$ with $n\le r$.
    The number of Boolean functions $f:\{0,1\}^n\to\{0,1\}$ that can be defined by $\Sigma_1$-formulas of size at most $r$ is at most $2^{r^{O(1)}}$.
\end{theorem}
\begin{proof}
    We repeat the counting argument from the proof of Hrube\v{s}~\cite[Theorem~2.2]{Hru20}. Let $f: \{0,1\}^{n} \to \{0,1\}$ be a Boolean function defined by a $\Sigma_{1}$-formula $F$ of size at most $r$. We associate to $F$ its skeleton $S_{F}$, where we replace its real constants by parameter variables $(y_1,\dots,y_r)$. We denote the resulting formula $S_F(x,y)$ as the \defn{skeleton} of $F(x)$.

     We first want to bound the number of Boolean functions that have the same skeleton. Namely, the Boolean function is determined exactly by $(F(\sigma))_{\sigma \in \{0,1\}^{n}}$. Thus, for any $a \in \mathbb{R}^{r}$, the Boolean function for specializing $S_{F}(x,y)$ to $y = a$ is determined by $(S_{F}(\sigma, a))_{\sigma \in \{0,1\}^{n}}$. By Lemma \ref{Hrubes-lemma},
    \[\#\{(S_{F}(\sigma,a)_{\sigma\in\{0,1\}^n}):a\in\R^r\}\le(2^{r^{O(1)}}2^n)^r\le2^{r^{O(1)}},\]
    using that $n\le r$. Hence, the number of Boolean functions defined by $\Sigma_{1}$-formulas of size at most $r$ with a given skeleton is at most $2^{r^{O(1)}}$.

    After adding a padding-symbol, every skeleton of size at most $r$ can be represented by a string of length $r$ over an alphabet of size $O(r)$ composed of the at most $r$ quantified variables, at most $r$ parameter variables, at most $r$ unquantified variables and constantly many connectives. Thus, the number of skeletons is at most $(O(r))^{r} \subseteq 2^{O(r \log(r))}$.
    
    Therefore, the total number of Boolean functions defined by $\Sigma_{1}$-formulas of size at most $r$ is at most the product of the number of potential skeletons and the number of Boolean functions with a given skeleton, which is  $2^{r^{O(1)}}\cdot2^{O(r\log(r))}\le2^{r^{O(1)}}$.
\end{proof}

\begin{proof}[Proof of Theorem~\ref{thm:counting}]
    Let $P\subseteq\R^n$ be a 0/1-polytope. 
    It can be uniquely associated with the Boolean function $f:\{0,1\}^n\to\{0,1\}$ with $f(\sigma)=1$ if and only if $\sigma\notin P$.
    By Proposition~\ref{prop:lb}, there is an existential $\Sigma_1$-formula $F_P(x)$ of size $r=O(s(s+n))\le O(s^2)$, which is true if and only if $x\notin P$.
    By Theorem~\ref{Hrubes}, there are at most $2^{r^{O(1)}}$ Boolean functions that can be defined by a $\Sigma_1$-formula of size at most $r$.
    Thus, there is some $c>0$, such that the number of such Boolean functions (and hence polytopes) is at most $2^{s^c}$.
\end{proof}

\section{Virtual Extension Complexity of Polygons}
\label{sec:polygons}

Due to slightly heavier notation, we start with some conventions for this section. 
If we have a list of vectors, we denote the $i$th vector in the list by $x^{(i)}$. 
In contrast, when we refer to the $i$th entry of a vector $x$, we denote it $x_{i}$. 
For a matrix $A$ with rows indexed by a set $S$ and columns by a set $T$, we denote the submatrix with rows $S' \subseteq S$ by $A_{S', \bullet}$ or with columns $T' \subseteq T$ by $A_{\bullet, T'}$. If we omit one index and just write $A_{k}$, that refers to the $k$th row of the matrix $A$.

Given a polytope $P = \{x: A x \leq b\}$ with vertex set $V(P)$, one can form a matrix $S$ with columns $b - Av$ for each vertex $v \in V(P)$. This matrix $S$ is called the \defn{slack matrix} of $P$. By feasibility of vertices, the slack matrix is always non-negative. The \defn{non-negative rank} of a non-negative matrix is the minimum number of non-negative rank $1$ matrices needed to write that matrix as a linear combination with non-negative coefficients. 

\begin{theorem}[\cite{YannakakisLemma}]
For any polytope $P$, $\xc(P)$ is the non-negative rank of its slack matrix.
\end{theorem}

This is the central tool in the theory of extension complexity. In general, non-negative rank and rank differ substantially as highlighted for example in Theorem 1.4 of \cite{LowDimExtensions}. However, in the case of rank $\leq 2$, it turns out they coincide:

\begin{theorem}[Theorem 4.1 of \cite{NonnegRank2}]
\label{thm:LowRank}
For non-negative matrices of rank at most $2$, their non-negative rank is equal to their rank.
\end{theorem}

An observation we rely on frequently is that for a non-negative matrix composed of two matrices
\[C=\begin{pmatrix} A \\ B\end{pmatrix},\]
the non-negative rank of $C$ is at most the non-negative rank of $A$ plus the non-negative rank of $B$ as one can just combine rank $1$ decompositions of $A$ and $B$ into a rank $1$ decomposition of $C$ by adding $0$ blocks.

For the upper bound, we are able to use a simpler version of the beautiful ideas from Kwan, Sauermann, and Zhao \cite{LowDimExtensions} for finding upper bounds for cyclic polygons. First, as is standard in the extension complexity literature, they use the equivalent formulation by Yannakakis in \cite{YannakakisLemma} as finding the non-negative rank of the slack matrix. They then split the vertices of the polygons into buckets of size $O(\sqrt{n})$ and then use this to decompose the slack matrix. However, making sure the geometry is tame enough to get their decomposition to work is both delicate and nontrivial. 

For us, since we use virtual extension complexity, we can instead Minkowski sum with a polygon $Q$ that is freely chosen. We choose $Q$ to have $O(\sqrt{n})$ many vertices and such that each normal cone of a vertex of $Q$ intersects the interior of the normal cones of at most $O(\sqrt{n})$ vertices of $P$. Then $\xc(Q)$ is $O(\sqrt{n})$, since it has $O(\sqrt{n})$ many vertices, so it suffices to bound the extension complexity of $P+Q$. In fact, extension complexity does not change under scaling, so it suffices to bound the extension complexity of $P+\lambda Q$ for $\lambda$ sufficiently large. The normal cones of vertices of $Q$ give a partition of the vertices of $P+\lambda Q$ into $O(\sqrt{n})$ pieces all of size $O(\sqrt{n})$, which we use to decompose the slack matrix much like Kwan, Sauermann, and Zhao. Analyzing the decomposition then yields our upper bound. 

The lower bound argument is a straightforward generalization of the transcendence degree based argument for lower bounding extension complexity of polygons from \cite{PolygonExtensions}.

Finally, we rely on some basic terminology regarding Minkowski sums and normal fans as may be found in Chapter $7$ of \cite{LecturesonPolytopes}. In particular, for a polygon, given a vertex one can construct a cone by taking the non-negative span of the two outer normals of the edges incident to that vertex called the \defn{normal cone} of that vertex. The set of all normal cones is called the normal fan. Given two polygons $P$ and $Q$, the normal fan of the Minkowski sum is the common refinement of the two normal fans. In particular, the set of rays of the cones will just be the union of the rays of all cones.

\subsection{The Upper Bound}

Our strategy to prove our upper bound is as follows. We are given as input a polygon $P$ with $n$ vertices, and we will Minkowski sum it with a sufficiently large scaling $\lambda Q$ of a polygon $Q$ with $O(\sqrt{n})$ vertices. Then $\vxc(P) \leq \xc(\lambda Q) + \xc(P+\lambda Q)$. Since $\lambda Q$ has $O(\sqrt{n})$ vertices, its extension complexity is at most $O(\sqrt{n})$. Thus, the core technical hurdle is to choose $Q$ in such a way that $\xc(P + \lambda Q)$ is $O(\sqrt{n})$.

The proof follows from a technical lemma that adapts the approach of Kwan, Sauermann, and Zhao for bounding the extension complexity of cyclic polygons. There they split the edges of the polygon into consecutive blocks of size $\lceil\sqrt{n}\rceil$. Then they use these blocks to find a decomposition of their matrix.

We will do the same here, except our advantage is that $Q$ is engineered. We can choose $Q$ freely such that each normal cone of a vertex of $Q$ intersects the interior of at most $O(\sqrt{n})$ normal cones of a vertex of $P$ on their interior. This will give a partition into blocks that we then use to decompose the matrix. The choice of $\lambda$ sufficiently large comes into play to ensure that each piece of the decomposition is non-negative. 

\begin{lem}
\label{lem:MinkSumBound}
Let $P$ and $Q$ be polygons with irredundant inequality descriptions $P= \{x \in \mathbb{R}^{2}: Ax \leq b\}$ and $Q = \{x \in \mathbb{R}^{2}: C x \leq d\}$ respectively, where $A$ has $m$ rows, $C$ has $n$ rows and no two rows of $A$ and $C$ are positive scalar multiples of one another. Then there exists $\lambda > 0$ such that
\[\xc(P+\lambda Q) \leq \ell + 3n,\]
where $\ell$ is the maximum number of normal cones of vertices of $P + \lambda Q$ all contained in the same normal cone of a vertex of $Q$. 
\end{lem}

\begin{proof}
 Let $R = P + \lambda Q$. Since no row of $A$ is a positive multiple of a row of $C$, the set of rays of the normal fan of $P$ is disjoint from the set of rays of the normal fan of $Q$. Since $R$ is a Minkowski sum of $P$ and $\lambda Q$ and Minkowski sum refines the normal fans, the inequalities defining $P+\lambda Q$ come from the rows of $A$ and $C$. Thus, we have
\[R = \left\{x \in \mathbb{R}^{2}: \begin{pmatrix} A \\ C\end{pmatrix} x \leq \begin{pmatrix} \tilde{b} \\ \tilde{d}\end{pmatrix}\right\}\]
for some $\tilde{b} \in \mathbb{R}^{m}$ and $\tilde{d} \in \mathbb{R}^{n}$. 

We rely on the non-negative rank formulation for extension complexity, so we will need the slack matrix of $R$. For this, we need to list the vertices of $R$. We will do so with a slight adjustment. 

The normal fan of $R$ is the common refinement of the normal fans of $P$ and $Q$. By definition of $\ell$, the normal cone of each vertex of $Q$ is subdivided into at most $\ell$ normal cones when Minkowski summed with $P$. In particular, we may write the vertices of $Q$ in cyclic order $q^{(1)}, q^{(2)}, \dots, q^{(n)}$. Then for each $q^{(j)}$, we may write the vertices of $P$ with normal cones that intersect the interior of the normal cone of $q^{(j)}$ in cyclic order as $p^{(j,1)}, \dots, p^{(j,\ell)}$. Note there may be fewer than $\ell$ vertices for which this is the case. It will be useful as a notational tool to just repeat the last vertex for last $\ell- k$ coordinates if there are only $k< \ell$ vertices. Then we let $r^{(i,j)} = \lambda q^{(i)} + p^{(i,j)}$ for all $i \in [n]$ and $j \in [\ell]$. The slack matrix has columns corresponding to vertices of $R$, where each column is 
\[s^{(i,j)} = \begin{pmatrix} \tilde{b} \\ \tilde{d}\end{pmatrix} -\begin{pmatrix} A \\C\end{pmatrix}r^{(i,j)}\] 
for each distinct vertex $r^{(i,j)}$. As noted in the preliminaries, $\xc(P+\lambda Q)$ is the non-negative rank of the slack matrix. Furthermore, if we add a copy of a column, the non-negative rank of the slack matrix does not change. Thus, we may index the columns by $[n] \times [\ell]$, where column $(i,j)$ corresponds to $s^{(i,j)}$. The rows of the slack matrix correspond to the rows of $A$ and rows of $C$. We will write the slack matrix as
\[S= \begin{pmatrix}S_{A} \\S_{C} \end{pmatrix},\]
where $S_{A}$ is the submatrix corresponding to rows of $A$ and $S_{C}$ is the submatrix corresponding to the rows of $C$.

Then $S_{C}$ has $n$ rows and thus has non-negative rank at most $n$. The non-negative rank of the slack matrix is at most the sum of the non-negative rank of $S_{A}$ and the non-negative rank of $S_{C}$. Therefore, to prove the result, it suffices to show the non-negative rank of $S_{A}$ is at most $\ell + 2n$. 

 For a vertex of $P + \lambda Q$, its slack is a weighted sum of its slack from $P$ and slack from $Q$. Since $\lambda$ is very large, the impact of the slack from $Q$ is more substantial. Our decomposition separates the impact of these slacks.
 
 Define $f: [m] \to [n]$ by 
 \[f(k) = \arg\max_{j \in [n]} A_{k} q^{(j)}.\]
 That is, $f(k)$ is the index of the unique vertex of $Q$ maximized by normal of the $k$th inequality of $P$. 

 Define a matrix $M$ with $m$ rows and $\ell n$ columns by 
 \[M_{k,(i,j)} =  s^{(f(k),j)}_{k},\]
 where $s^{(f(k),j)}_{k}$ is the slack at the $k$th inequality of $A$ at the vertex $r^{(f(k),j)}$. Note that, by construction, for all $i \neq i' \in [n]$,
 \[M_{k,(i,j)} = s^{(f(k),j)}_{k}= M_{k,(i',j)}.\]
 Thus, $M$ has at most $\ell$ distinct columns and all entries are entries of $S_{A}$ and thus non-negative. Hence, $M$ has non-negative rank at most $\ell$.

 Furthermore, $M$ and $S_{A}$ have the same dimensions, so $N = S_{A} - M$ is well-defined and $S_{A} = M + N$. Then to complete the proof, since $M$ has non-negative rank at most $\ell$, it suffices to show that $N$ is non-negative and has non-negative rank at most $2n$. For non-negativity, observe that   
 \begin{align*}
     N_{k, (i,j)} &= (S_{A})_{k,(i,j)} - M_{k,(i,j)} \\
     &= s^{(i,j)}_{k} - s^{(f(k),j)}_{k} \\
     &= (\tilde{b}_{k}- A_{k}(\lambda q^{(i)}+p^{(i,j)})) -(\tilde{b}_{k} - A_{k}(\lambda q^{(f(k))} +  p^{(f(k),j)})) \\
     &= A_{k}(\lambda(q^{(f(k))}-q^{(i)})+p^{(f(k),j)}-p^{(i,j)}) \\
     &= \lambda (A_{k}q^{(f(k))} - A_{k}q^{(i)})+A_{k}(p^{(f(k),j)} - p^{(i,j)}).
 \end{align*}
In particular, if $i = f(k)$, one can see from this computation that 
\[N_{k,(i,j)} = \lambda (A_{k}q^{(f(k))} - A_{k}q^{(f(k))})+A_{k}(p^{(f(k),j)} - p^{(f(k),j)}) = 0.\] 
 
To ensure non-negativity when $i \neq f(k)$, we need to use that $\lambda$ is sufficiently large. By definition of $f$, $A_{k}$ is maximized uniquely at $q^{(f(k))}$ over $Q$. Thus, $A_{k}q^{(f(k))} - A_{k}q^{(i)} > 0$ for all $i \neq f(k)$. Thus, for all $\lambda$ sufficiently large,
 \[\lambda (A_{k}q^{(f(k))} - A_{k}q^{(i)}) > -A_{k}(p^{(f(k),j)} - p^{(i,j)}), \]
 which implies that for all $\lambda$ sufficiently large
 \[N_{k,(i,j)} = \lambda (A_{k}q^{(f(k))} - A_{k}q^{(i)})+A_{k}(p^{(f(k),j)} - p^{(i,j)}) > 0.\]
 Finally, all that remains is to show $N$ has non-negative rank at most $2n$. To do this for each $t \in [n]$, let $N(t)$ denote the submatrix of $N$ with rows in $f^{-1}(t)$. It suffices to show then that $N(t)$ has non-negative rank at most $2$. If a non-negative matrix has rank at most $2$, by Theorem \ref{thm:LowRank} it has non-negative rank at most $2$.
 
 Note that for all $t \in [n]$ with $k \in f^{-1}(t)$
 \[N_{k,(i,j)} = A_{k}(\lambda(q^{(t)}-q^{(i)})+p^{(t,j)}-p^{(i,j)})\]
 Thus, for all $k \in f^{-1}(t)$,
 \[N(t)_{k,(i,j)} = A_{k}(\lambda(q^{(t)}-q^{(i)})+p^{(t,j)}-p^{(i,j)}). \]
 Let $O(t)$ be the $2 \times n\ell$ matrix with columns  
 \[O_{\bullet, (i,j)}(t) = \lambda(q^{(t)}-q^{(i)})+p^{(t,j)}-p^{(i,j)}.\]
 Then, by definition of matrix multiplication,
 \[N(t) = A_{f^{-1}(t),\bullet} O(t).\]
 Note that $A_{f^{-1}(t),\bullet}$ is an $|f^{-1}(t)| \times 2$ matrix, and $O(t)$ is a $2 \times n\ell$ matrix, so $N(t)$ has rank at most $2$. Any non-negative matrix of rank at most $2$ has non-negative rank at most $2$ by Theorem \ref{thm:LowRank}, so $N(t)$ has non-negative rank at most $2$. This completes the proof, but we may trace back the reasoning explicitly.

 Note $N(t)$ has non-negative rank at most $2$, so since the rows of $N$ are partitioned by the union over all $n$ choices of $t$ of $N(t)$, $N$ has non-negative rank at most $2n$. As already observed $S_{A} = M + N$, where $M$ has non-negative rank at most $\ell$, so $S_{A}$ has non-negative rank at most $\ell + 2n$. Furthermore, $S_{C}$ has non-negative rank at most $n$, so the slack matrix $S = \begin{pmatrix} S_{A} \\ S_{C}\end{pmatrix}$ has non-negative rank at most $\ell + 3n$. The non-negative rank of $S$ is the extension complexity of $R = P + \lambda Q$. Therefore, $P + \lambda Q$ has extension complexity at most $\ell + 3n$ as desired. 
\end{proof}

The upper bound of Theorem \ref{thm:main} follows as an application of this lemma.

\begin{proof}[Proof of Theorem \ref{thm:main} Upper Bound]
Let $P$ be a polygon with $n$ vertices. Label the vertices cyclically $1,2, \dots, n$, and for each vertex with index a multiple of $\lceil \sqrt{n} \rceil$, select a point in the interior of its normal cone on the unit circle. If the convex hull of all such points does not contain the origin in its interior, add three generically chosen points on the unit circle such that their convex hull contains the origin in its interior that all lie in the interior of some normal cones of vertices of $P$. 

Let $Q$ be the polar of the resulting point configuration. Then the edges of $Q$ have normals all in the interior of normal cones of $P$. Furthermore, $Q$ has at most $n/\lceil \sqrt{n}\rceil + 3 \leq \sqrt{n} +3$ many edges. By construction of the choice of normals of $Q$, the cone between any two consecutive normals only intersects the interior of at most $O(\sqrt{n})$ many cones. That is, $\ell$ as from the statement of Lemma \ref{lem:MinkSumBound} is $O(\sqrt{n})$.  Hence, by Lemma \ref{lem:MinkSumBound}, there exists $\lambda > 0$ so large that the extension complexity of $P + \lambda Q$ is at most $ O(\sqrt{n})$. Since $\lambda Q$ has at most $\sqrt{n} + 3$ many vertices, its extension complexity is at most $O(\sqrt{n})$ as well. Therefore, since the virtual extension complexity of $P$ is at most the sum of the extension complexity of $\lambda Q$ and $P + \lambda Q$, we have $\vxc(P)$ is in $O(\sqrt{n})$.
\end{proof}

\subsection{The Lower Bound}

The lower bound argument is similar enough to the lower bound argument for extension complexity using algebraically independent parameters from \cite{PolygonExtensions} that we recall it here. Namely, for a polygon $P$ with $n$ vertices, consider an extended formulation $Q$ of $P$ with $m$ facets in $d$ dimensions. Then $d \leq m-1$, since any $d$-dimensional polytope has at least $d+1$ facets. Therefore, $Q$ is described by a system of linear inequalities $Q = \{\mathbf{x}: Ax \leq b\}$, where $A$ has $m$ rows and at most $m-1$ columns. Furthermore, since each vertex of $Q$ is a basic feasible solution by standard results regarding the simplex method as in Chapter 4 of \cite{GartnerBook} and by Cramer's rule, each vertex of $Q$ is a rational function of the entries of $A$ and $b$. Since $Q$ is an extended formulation of $P$, there is a $2 \times d$ matrix $D$ such that $DQ = P$. In particular, every vertex of $P$ is a rational function of the entries of $D$, $A$, and $b$. The total number of entries of $D, A$, and $b$ is $2d + md + m \leq 2(m-1)+m(m-1) +m \leq (m+1)^{2}.$ Therefore, if the vertices of $P$ have algebraically independent coordinates, they must be described by at least $2n$ algebraically independent parameters in the extended formulation. That is, $(m+1)^{2} \geq 2n$, so $m \geq \sqrt{2n}-1$. That is why the extension complexity of such a polygon is $\Omega(\sqrt{n})$. We push this argument slightly further to obtain the lower bound.

\begin{proof}[Theorem \ref{thm:main} Lower Bound]

Let $P$ be a polygon with $n$ vertices. Each vertex is described by two coordinates, and assume that all $2n$ coordinates are algebraically independent over $\mathbb{Q}$. Suppose that $P + Q = R$. Then each vertex $p$ of $P$ is of the form $p = r - q$ for vertices $q$ and $r$ of $Q$ and $R$ respectively. By the previous argument, any extended formulation of $Q$ with $s_Q$ facets is described by at most $(s_Q+1)^{2}$ parameters, and any extended formulation of $R$ with $s_R$ facets is described by at most $(s_R+1)^{2}$ parameters. Any vertex of $P$ is a rational function of those parameters. Thus, $(s_Q+1)^{2} + (s_R+1)^{2} \geq 2n$, so $\max((s_Q+1)^{2}, (s_R+1)^{2}) \geq n$. It follows then that $s_Q+s_R \geq \max(s_Q,s_R) \geq \sqrt{n} - 1$. Thus, the total number of facets of the extended formulations of $Q$ and $R$ is at least $\sqrt{n} - 1$. Hence, $\vxc(P)$ is in $\Omega(\sqrt{n})$. 
\end{proof}

\section*{AI Disclosure}
The proofs of all three of our main results arose from multiple discussions with different versions of ChatGPT (ranging from 5.6 Sol to 6 Pro). As is typical currently, the proofs it provided were quite difficult to parse, and so we made efforts to digest and simplify it. In particular, the exposition was entirely written by the authors, though ChatGPT was used at multiple stages of the editing process to check for typos and correctness.

\section*{Acknowledgments}
We thank Daniel Dadush, Samuel Fiorini, and Stefan Weltge for recent discussions on virtual extension complexity, e.g., during the Cargese Workshop on Combinatorial Optimization 2026.

\bibliographystyle{amsplain}
\bibliography{bibliography.bib}

\end{document}